\documentclass[11pt]{article}

\usepackage[T1]{fontenc}
\usepackage{lmodern}
\usepackage{amsmath,amssymb,amsthm}
\usepackage{booktabs,array}
\usepackage[margin=1in]{geometry}
\usepackage{xcolor}
\usepackage[colorlinks=true,linkcolor=blue!55!black,
  citecolor=green!40!black,urlcolor=blue!60!black]{hyperref}
\hypersetup{
  pdftitle={Hamilton Starters and Path Decompositions in Directed Circulants},
  pdfauthor={Jianwei Jiang and Chunhua Yang},
  pdfsubject={Exact four-layer and asymptotic three-layer constructions}}

\newtheorem{theorem}{Theorem}[section]
\newtheorem{lemma}[theorem]{Lemma}
\newtheorem{proposition}[theorem]{Proposition}
\newtheorem{corollary}[theorem]{Corollary}
\theoremstyle{definition}
\newtheorem{definition}[theorem]{Definition}

\theoremstyle{remark}
\newtheorem{remark}[theorem]{Remark}
\newtheorem*{externalresult}{External result}

\newcommand{\Z}{\mathbb Z}
\newcommand{\Cayto}{\operatorname{Cay}^{\rightarrow}}
\newcommand{\Cqr}[2]{\mathcal C_{#1}(#2)}
\newcommand{\Cfour}[1]{\mathcal C_4(#1)}
\newcommand{\Cthree}[1]{\mathcal C_3(#1)}
\newcommand{\Oarc}[2]{\mathcal O_{#1,#2}}
\newcommand{\pn}{\operatorname{pn}}
\newcommand{\succL}{\operatorname{succ}_{L}}

\title{Hamilton Starters and Path Decompositions in Directed Circulants}
\author{
Jianwei Jiang \qquad Chunhua Yang\thanks{Corresponding author.}\\[0.4em]
School of Mathematics and Statistics, Weifang University, Weifang, Shandong, China\\
\texttt{jianweijiangwf@sina.com} \qquad
\texttt{chunhuayangwf@sina.com}
}
\date{}

\begin{document}
\maketitle

\begin{abstract}
For integers \(q\geq3\) and \(r\geq1\), let
\(\Cqr{q}{r}:=\Cayto(\Z_{qr},\{1,\ldots,r\})\).  A Hamilton cycle \(H\) is
called a \(q\)-layer balanced Hamilton starter if, for every step
\(s\in\{1,\ldots,r\}\) and every residue class \(j\in\Z_q\), it contains
exactly one arc \(x\to x+s\) with \(x\equiv j\pmod q\).  Its translates
\[
H,\ H+q,\ \ldots,\ H+(r-1)q
\]
then form a Hamilton decomposition of \(\Cqr{q}{r}\).  A \(q\)-layer balanced
Hamilton starter is called chain-compatible if one arc can be selected from
each translated Hamilton cycle so that the selected arcs form a simple
directed path.  In the constructions below, this compatible deletion chain is
prescribed as \(0\to2\to4\to\cdots\to2r\).
For \(q=4\), a chain-compatible starter is obtained explicitly whenever
\(r\equiv1\pmod4\) and \(r\geq9\), while for \(q=3\) one exists for all
sufficiently large \(r\equiv5\pmod6\). The proofs are constructive: the four-layer case uses an \(ABAB\) step word, while in the three-layer case a directed rotational terrace is lifted to a \(3\)-layer balanced directed \(1\)-factor and a fixed four-arc trade joins its two cycles. Deleting the unique prescribed-path arc from each translated Hamilton cycle gives
\(r\) Hamilton paths; together with the prescribed path, these form an optimal
decomposition of the arc set into \(r+1\) directed paths in both cases.
\end{abstract}
\section{Introduction}
\label{sec:introduction}

Alspach and Pullman \cite{AlspachPullman1974} introduced the path number
\(\pn(D)\) of a digraph \(D\), the minimum number of simple directed paths
whose arc sets partition \(A(D)\).  The tournament case was studied in
\cite{AlspachMasonPullman1976,GiraoGranetKuhnLoOsthus2023,LoPatelSkokanTalbot2020};
see also \cite{PatelYildiz2026} for oriented graphs.  For dense regular
digraphs, K\"uhn and Osthus \cite{KuhnOsthus2013} proved Hamilton
decomposability of sufficiently large regular robust outexpanders of linear
degree, while Gir\~ao et al. \cite[Theorem~5.2]{GiraoGranetKuhnLoOsthus2023}
proved the corresponding path-number equality.  For the present circulants, a direct sliding-window argument verifies the
robust expansion needed to apply Gir\~ao et al. for fixed
\(q\in\{3,4\}\) and all sufficiently large \(r\); see
Appendix~\ref{app:robust-expansion}.  Thus ordinary asymptotic Hamilton
decomposability and the corresponding path-number equality are not the
structural focus here.

For integers \(q\geq3\) and \(r\geq1\), write
\([r]:=\{1,\ldots,r\}\) and let
\[
  \Cqr{q}{r}:=\Cayto\bigl(\Z_{qr},[r]\bigr).
\]
This \(r\)-regular digraph has \(qr^2\) arcs, so
\(\pn(\Cqr{q}{r})\geq\lceil qr^2/(qr-1)\rceil=r+1\).  These are
consecutive-step circulant digraphs, a special case of the consecutive-\(d\)
digraphs studied for Hamiltonicity by Du and Hsu \cite{DuHsu1989}.  Cyclic
starters and difference constructions are classical tools for cycle
factorizations: directed starter cycles and cyclic development appear in
Jordon and Morris \cite{JordonMorris2009}, relative starters with subgroup
transversals in \cite{PasottiPellegrini2018}, and base-\(q\) differences in
Kadri and \v{S}ajna \cite{KadriSajna2025}.  In our setting, translation by
\(q\) has \(qr\) arc orbits indexed by \((j,s)\in\Z_q\times[r]\), where
\(j\) is the tail residue modulo \(q\) and \(s\) the step.  A Hamilton cycle
\(H\) meeting each orbit once is a \(q\)-layer balanced Hamilton starter,
equivalently a Hamiltonian orbit transversal; its translates
\(H,H+q,\ldots,H+q(r-1)\) form a Hamilton decomposition.  Thus the
development mechanism is standard; the difficulty is to construct such a
Hamiltonian orbit transversal inside the restricted connection set \([r]\).
Section~\ref{sec:layered} formalizes this viewpoint.

We further require the starter to be chain-compatible: one arc from each
developed Hamilton cycle can be selected so that the selected arcs form a
simple directed path.  In both constructions, this path is prescribed to be
the arithmetic step-\(2\) path
\begin{equation}\label{eq:intro-step2-chain}
  P_r:\qquad 0\to2\to4\to\cdots\to2r.
\end{equation}
Thus \(P_r\) is a compatible deletion chain.  This is a 1-orthogonality
condition in the sense of Liu
\cite{Liu2009}, realized by the Hamilton decomposition developed from the
same Hamiltonian orbit transversal.  Deleting these \(r\) chain arcs leaves
\(r\) Hamilton paths, while the deleted arcs form \(P_r\). Together,
these \(r+1\) paths attain the lower bound above.

For \(q=4\), every \(r\equiv1\pmod4\) with \(r\geq9\) admits an explicit
4-layer balanced Hamilton starter in \(\Cfour{r}\).  In the standard terrace
terminology of Ollis~\cite[Section~2]{Ollis2012}, recalled in
Section~\ref{sec:terraces}, a normalized directed \(R^*\)-terrace of
\(\Z_r\) yields an \(ABAB\) step word whose prefix structure gives
Hamiltonicity, four-layer balance, and the compatible deletion chain
\eqref{eq:intro-step2-chain}.  Consequently,
\(\pn(\Cfour{r})=r+1\); see Theorems~\ref{thm:four-exact}
and~\ref{thm:four-path}.

For \(q=3\), there is a non-explicit constant \(r_0\) such that every
\(r\equiv5\pmod6\), \(r\geq r_0\), admits a 3-layer balanced Hamilton starter
in \(\Cthree{r}\).  A directed rotational terrace is lifted to a 3-layer
balanced directed 1-factor with a long cycle and a 3-cycle.  A fixed four-arc
trade joins the two cycles while preserving balance and leaving the three
step-\(2\) starter arcs unchanged; these arcs yield the compatible deletion
chain \eqref{eq:intro-step2-chain}.  Thus
\(\pn(\Cthree{r})=r+1\).  The numerical equality is only an asymptotic
consequence; the structural content is the specified cyclic Hamilton
decomposition and its compatible deletion chain.  See
Theorem~\ref{thm:three-layer-chain-compatible} and Corollary~\ref{cor:three-layer-path-number}.  The
prescribed-path extension for directed rotational terraces is derived in
Section~\ref{sec:prescribed} from
\cite[Theorem~6.9]{MuyesserPokrovskiy2025}.

\section{Layered directed circulants and balanced Hamilton starters}
\label{sec:layered}

All cyclic groups are written additively. For any $m\geq1$, when an element of $\Z_m$ is used
as an integer representative, we take its least nonnegative representative in
$\{0,\ldots,m-1\}$. When a nonzero element of $\Z_r$ is used as a positive
step, $[x]_r$ denotes its representative in $\{1,\ldots,r-1\}$. Recall
that $[r]=\{1,\ldots,r\}$, and write
\begin{equation}\label{eq:def-cqr}
  \Cqr{q}{r}:=\Cayto(\Z_{qr},[r]).
\end{equation}
The symbol $r$ itself denotes the additional allowed step in $[r]$; although it is $0$ modulo $r$, it is a nonzero step in $\Z_{qr}$. An arc $x\to x+s$ has \emph{tail} $x$ and \emph{step} $s$.  Since
$r<qr/2$ for $q\geq3$, no reverse pair of allowed arcs occurs, so the
digraph is oriented.

\begin{definition}\label{def:starter}
For $j\in\Z_q$, the \emph{$j$-th layer} of $\Cqr{q}{r}$ is the set of
vertices congruent to $j$ modulo $q$. A spanning subdigraph $F$ of
$\Cqr{q}{r}$ is \emph{$q$-layer balanced} if, for every
$(j,s)\in\Z_q\times[r]$, it contains exactly one arc of step $s$ whose tail
lies in the $j$-th layer. A \emph{$q$-layer balanced Hamilton starter} is a
Hamilton cycle of $\Cqr{q}{r}$ that is $q$-layer balanced.
\end{definition}

Let $\tau_q:x\mapsto x+q$. Let \(G_q=\langle\tau_q\rangle\cong\Z_r\) be the cyclic group of order \(r\) generated by \(\tau_q\). It acts naturally on the arc set. For
$j\in\Z_q$ and $s\in[r]$, define
\begin{equation}\label{eq:arc-orbit}
  \Oarc{j}{s}:=
  \{x\to x+s:x\equiv j\pmod q\}.
\end{equation}

\begin{proposition}\label{prop:arc-orbits}
The sets $\Oarc{j}{s}$ are precisely the $qr$ orbits of $G_q$ on
$A(\Cqr{q}{r})$.  Every orbit has size $r$, and the action on it is regular.
\end{proposition}

\begin{proof}
Translation by \(q\) preserves both the step of an arc and the residue class modulo \(q\) of its tail. Hence, if \(e=x\to x+s\in\Oarc{j}{s}\), then the entire \(G_q\)-orbit of \(e\) is contained in \(\Oarc{j}{s}\). Conversely, let \(e'=y\to y+s\in\Oarc{j}{s}\). Since \(x\equiv y\equiv j\pmod q\), there exists \(k\in\Z_r\) such that \(y-x\equiv qk\pmod{qr}\). Therefore \(\tau_q^k(e)=(x+qk)\to(x+qk+s)=y\to y+s=e'\). Thus every arc of \(\Oarc{j}{s}\) lies in the orbit of \(e\), so \(\Oarc{j}{s}\) is exactly one \(G_q\)-orbit. If \(\tau_q^k\) fixes an arc, then \(qk\equiv0\pmod{qr}\), and hence \(k\equiv0\pmod r\). Thus the stabilizer of every arc is trivial, so the action on each orbit is regular. In particular, every orbit has size \(|G_q|=r\). Finally, there are \(q\) choices for \(j\) and \(r\) choices for \(s\), so the \(qr\) sets \(\Oarc{j}{s}\) are pairwise disjoint and cover all \(qr^2\) arcs of \(\Cqr{q}{r}\). Hence they are precisely all \(G_q\)-orbits on \(A(\Cqr{q}{r})\).
\end{proof}

Thus a $q$-layer balanced Hamilton starter is precisely a Hamiltonian
transversal of the $G_q$-arc orbits.

\begin{lemma}\label{lem:development}
If $H$ is a $q$-layer balanced Hamilton starter, then
\begin{equation}\label{eq:developed-family}
  H,H+q,H+2q,\ldots,H+q(r-1)
\end{equation}
are pairwise arc-disjoint Hamilton cycles whose union is
$A(\Cqr{q}{r})$.
\end{lemma}

\begin{proof}
Translation is an automorphism, so every cycle in
\eqref{eq:developed-family} is Hamilton.  The starter meets every orbit
$\Oarc{j}{s}$ once by Definition~\ref{def:starter}.  Regularity in
Proposition~\ref{prop:arc-orbits} says that its $r$ translates meet every
arc of that orbit once.  Applying this to all $qr$ orbits proves both
disjointness and coverage.
\end{proof}

\begin{figure}[ht]
\centering
\fbox{\begin{minipage}{0.91\linewidth}
\centering
\textbf{Development schematic.}\quad
$\Oarc{j}{s}=\{e,\tau_qe,\ldots,\tau_q^{r-1}e\}$
\[
\begin{array}{c|cccc}
\text{cycle}&H&H+q&\cdots&H+q(r-1)\\ \hline
\text{arc in }\Oarc{j}{s}&e&\tau_qe&\cdots&\tau_q^{r-1}e
\end{array}
\]
The same row occurs for each of the $qr$ pairs $(j,s)$.  A $q$-layer balanced
Hamilton starter is exactly a Hamiltonian orbit transversal.
\end{minipage}}
\caption{One starter develops to the full arc decomposition.}
\label{fig:development}
\end{figure}

For the present circulants, we define the \emph{winding number} of a closed directed walk $H$ in $\Cqr{q}{r}$ as follows.  Let $s(e)\in[r]$ denote the positive step of an arc $e$ of $H$, and set
\[
  \omega(H):=\frac{1}{qr}\sum_{e\in H}s(e),
\]
where arcs are counted with multiplicity.

\begin{proposition}\label{prop:winding}
If a $q$-layer balanced Hamilton starter exists, then $r$ is odd and every
such starter has winding number
\begin{equation}\label{eq:forced-winding}
  \omega(H)=\frac{r+1}{2}.
\end{equation}
\end{proposition}

\begin{proof}
Every step $s\in[r]$ occurs exactly $q$ times.  The total positive
displacement is therefore
\[
 q\sum_{s=1}^r s=\frac{qr(r+1)}2.
\]
Since $H$ is closed in $\Z_{qr}$, its total positive displacement is an integer multiple of $qr$. Thus $qr$ divides $qr(r+1)/2$, which holds exactly when $r$ is odd. In that case the quotient, and hence the winding number, is $(r+1)/2$.
\end{proof}

\section{Directed rotational terraces and rainbow Hamilton cycles}\label{sec:terraces}

Throughout this section, let \(r\geq2\).

For background and terminology on directed rotational terraces and directed
\(R^*\)-terraces, see Ollis~\cite[Section~2]{Ollis2012}. We use the
following cyclic form.

\begin{definition}\label{def:terrace}
A \emph{directed rotational terrace} of $\Z_r$ is a cyclic ordering
\[
  T=(t_0,t_1,\ldots,t_{r-2})
\]
of $\Z_r\setminus\{0\}$ such that the cyclic differences
\begin{equation}\label{eq:terrace-diffs}
  d_i=[t_{i+1}-t_i]_r\qquad(i\in\Z_{r-1})
\end{equation}
also run through $1,\ldots,r-1$ exactly once.  It is a directed
$R^*$-terrace if $t_i=t_{i-1}+t_{i+1}$ for some cyclic index $i$. The cyclic word $D(T):=(d_0,d_1,\ldots,d_{r-2})$ is called the \emph{difference word} of $T$.
\end{definition}
A directed \(R^*\)-terrace has an additional local relation that will be
used in the four-layer construction.

Let $\Gamma_r$ be the complete directed graph on
$\Z_r\setminus\{0\}$. Colour $x\to y$ by the nonzero group element
$y-x\in\Z_r$; when written as a positive integer, this colour is represented
by $[y-x]_r\in\{1,\ldots,r-1\}$. A directed path or cycle is \emph{rainbow}
if its arc colours are pairwise distinct.

\begin{proposition}\label{prop:rainbow}
The following objects are equivalent:
\begin{enumerate}
  \item a directed rotational terrace of $\Z_r$;
  \item a rainbow Hamilton cycle in $\Gamma_r$ using every nonzero colour.
\end{enumerate}
The equivalence preserves the cyclic order of the entries and identifies
the terrace difference word with the colour sequence.
\end{proposition}

\begin{proof}
Given a terrace $T=(t_0,\ldots,t_{r-2})$, its entries are exactly the
$r-1$ vertices of $\Gamma_r$, and Definition~\ref{def:terrace} says that
the colours of its cyclic edges are exactly $\Z_r\setminus\{0\}$.  Thus it
is the required rainbow Hamilton cycle.  Conversely, list the vertices of
such a cycle in cyclic order.  Hamiltonicity ensures that every nonzero group
element appears once, while the rainbow condition and the $r-1$ available colours
force the consecutive differences to list every nonzero element once.
This is precisely a directed rotational terrace.
\end{proof}

For a finite word $X=(x_0,\ldots,x_{h-1})$ of length $h$ and an integer
$M\geq1$, define its prefix sums and modular tail-prefix set by
\begin{equation}\label{eq:prefixes}
  \Sigma_i(X):=\sum_{u=0}^{i-1}x_u\quad(0\leq i\leq h),
  \qquad
  U_M(X):=\{\Sigma_i(X)\bmod M:0\leq i<h\}.
\end{equation}
The values $\Sigma_i(X)$ with $0\leq i<h$ are called the \emph{tail prefix sums} of $X$; they are the tails of the successive arcs generated by the word. The terminal prefix $\Sigma_h(X)$ is omitted from $U_M(X)$. We call a length-\(M\) word \(X\) a \emph{Hamilton step word} on \(\Z_M\) if
\[
\Sigma_M(X)=0\quad\text{in }\Z_M
\qquad\text{and}\qquad
U_M(X)=\Z_M.
\] We will use the terrace order in the three-layer lift and
its difference word in the four-layer construction.

\begin{lemma}\label{lem:symmetries}
If $T$ is a directed rotational terrace of $\Z_r$, then cyclic rotation of
$T$, multiplication of every entry by a unit $c\in\Z_r^\times$, and
reversal followed by negation all produce directed rotational terraces.
The $R^*$ property is preserved by rotation and unit multiplication.
\end{lemma}

\begin{proof}
Rotation changes only the chosen origin of the cyclic word.  Unit
multiplication permutes both the nonzero entries and their nonzero
differences.  Reversal negates every oriented difference; the additional
negation restores the same difference convention.  Finally, the identity
$t_i=t_{i-1}+t_{i+1}$ is homogeneous and therefore is preserved under
multiplication by a unit, while a rotation only moves its index.
\end{proof}

\section{The exact four-layer construction}\label{sec:four}

Throughout this section $r\equiv1\pmod4$ and $r\geq9$.

\subsection{Normalizing the \texorpdfstring{\(R^*\)}{R*}-terrace}

The four-layer construction begins with the following existence theorem for directed $R^*$-terraces.

\begin{theorem}\label{thm:rstar}
For every odd $r>5$, the cyclic group $\Z_r$ has a directed $R^*$-terrace.
\end{theorem}

This existence result is due to Ollis~\cite[Lemma~2]{Ollis2012}.

Choose a directed \(R^*\)-terrace of \(\Z_r\). Cyclically rotate it, and denote the resulting terrace by
\[
P=(P_0,P_1,\ldots,P_{r-2}).
\]
Let its cyclic difference word be
\[
D=(d_0,d_1,\ldots,d_{r-2}),\qquad d_i=[P_{i+1}-P_i]_r,
\]
where the indices are taken cyclically. By the \(R^*\) relation and the choice of rotation,
\(
P_1=P_0+P_2
\)
in \(\Z_r\). Hence the first two differences are
\[
a:=d_0=P_1-P_0=P_2,\qquad b:=d_1=P_2-P_1=-P_0.
\]
They are nonzero, and \(a\neq-b\) because \(P_2\neq P_0\).

For the parity argument used later to obtain four-layer balance, we first normalize the two distinguished differences to have odd representatives.

\begin{lemma}\label{lem:odd-normalization}
Let $r$ be odd and let $a,b\in\Z_r\setminus\{0\}$ satisfy $a\neq-b$. There is a unit $c\in\Z_r^\times$ for which the least positive representatives of $ca$ and $cb$ are both odd.
\end{lemma}

\begin{proof}
For $x\in\{1,\ldots,r-1\}$ put $x_t=[2^t x]_r$. Since $x_t\equiv2x_{t-1}\pmod r$ and $2\leq2x_{t-1}\leq2r-2$, there is a unique $e_t\in\{0,1\}$ such that
\begin{equation}\label{eq:doubling-recurrence}
x_t=2x_{t-1}-re_t.
\end{equation}
Since $r$ is odd, reducing \eqref{eq:doubling-recurrence} modulo $2$ shows that $e_t$ is the parity of $x_t$. Rewriting \eqref{eq:doubling-recurrence} as $x_{t-1}/r=e_t/2+x_t/(2r)$ and iterating gives, for every $m\geq1$,
\[
\frac{x}{r}=\sum_{t=1}^{m}\frac{e_t}{2^t}+\frac{x_m}{2^m r}.
\]
Since $0<x_m<r$, the final term tends to $0$ as $m\to\infty$, and hence
\begin{equation}\label{eq:binary-determination}
\frac{x}{r}=\sum_{t\geq1}\frac{e_t}{2^t}.
\end{equation}
Therefore the parity sequence of the representatives $[2^t x]_r$ determines $x$.

Suppose that no unit makes both $a$ and $b$ odd. If some unit $c$ made both representatives even, then $-c$ would make both odd, since $[-y]_r=r-[y]_r$ reverses parity. Hence $[ca]_r$ and $[cb]_r$ have opposite parity for every unit $c$. Equivalently, $[ca]_r$ and $[c(-b)]_r$ have the same parity for every unit $c$. Since $r$ is odd, $2^t\in\Z_r^\times$ for every $t\geq1$. Taking $c=2^t$ and applying \eqref{eq:binary-determination} yields $a=-b$, a contradiction.
\end{proof}

By Lemma~\ref{lem:symmetries}, unit multiplication preserves both the terrace and its $R^*$ relation. After applying such a unit and relabelling the resulting terrace, difference word, and distinguished differences again as $P,D,a,b$, we may assume that the positive representatives $a,b\in\{1,\ldots,r-1\}$ are odd. Since $a=d_0$ and $b=d_1$, the cyclic difference word has the form
\begin{equation}\label{eq:D-decomposition}
D=(a,b,D'),
\end{equation}
where $D'$ lists all remaining nonzero residues.
\subsection{Complementary prefix lifts}

Define two permutations of $[r]$ and their common sum by
\begin{equation}\label{eq:A-B}
  A=(a,r,b,D'),\qquad B=(b,r,a,D'),
  \qquad S=\frac{r(r+1)}2,
\end{equation}
and set
\begin{equation}\label{eq:ABAB}
  \mathcal W=ABAB.
\end{equation}
We refer to the initial three-letter blocks $(a,r,b)$ and $(b,r,a)$ as the two local $R^*$-port blocks. The tail prefix sums of $D$ modulo $r$ are
\begin{equation}\label{eq:D-prefixes}
 \{P_i-P_0:0\leq i\leq r-2\}=\Z_r\setminus\{b\}.
\end{equation}
Indeed, the entries $P_i$ run through $\Z_r\setminus\{0\}$, so the only missing residue is $-P_0=b$. Moreover, $a=P_1-P_0$ is among the tail prefix sums.

\begin{lemma}\label{lem:prefix-multiplicity}
Modulo \(r\), the tail prefix sums of \(A\) contain \(a\) twice and omit \(b\), while those of \(B\) contain \(b\) twice and omit \(a\). Every other residue occurs exactly once in both lists.
\end{lemma}

\begin{proof}
Inserting $r$ immediately after the first letter $a$ of $D$ duplicates
the current prefix $a$ without changing residues modulo $r$.  This proves
the assertion for $A$ using \eqref{eq:D-prefixes}.  For $B$, compare the
initial prefix lists
\[
  0,a,a,a+b\qquad\text{and}\qquad0,b,b,a+b
\]
of $(a,r,b)$ and $(b,r,a)$.  Thereafter both words traverse the same suffix
$D'$ from the same accumulated sum.
\end{proof}

\begin{lemma}\label{lem:prefix-lifts}
One has
\begin{equation}\label{eq:prefix-partition}
  U_{2r}(A)\mathbin{\dot\cup}\bigl(r+U_{2r}(B)\bigr)=\Z_{2r}.
\end{equation}
\end{lemma}

\begin{proof}
The two occurrences of the tail prefix sum \(a\) in \(A\) lie immediately before and after the inserted \(r\), so they are the two distinct lifts \(a\) and \(a+r\). Likewise, the two occurrences of the tail prefix sum \(b\) in \(B\) are the two lifts \(b\) and \(b+r\). For every \(u\notin\{a,b\}\), the corresponding tail prefix sum in \(A\) equals the corresponding tail prefix sum in \(B\) modulo \(2r\): after their three-letter port blocks, both words have accumulated \(a+b+r\) and follow the same suffix \(D'\). Each residue modulo \(r\) has exactly two lifts modulo \(2r\), differing by \(r\); adding \(r\) to the tail prefix sum in \(B\) supplies the complementary lift. Thus each element of \(\Z_{2r}\) occurs exactly once in \eqref{eq:prefix-partition}.

\end{proof}

For $u\notin\{a,b\}$, let $p_u\in\Z_{2r}$ denote the common residue modulo $2r$ of the corresponding tail prefix sums in $A$ and $B$. Figure~\ref{fig:prefix-lifts} summarizes this lift pattern.

\begin{figure}[ht]
\centering
\fbox{\begin{minipage}{0.92\linewidth}
\centering
\textbf{Prefix-lift schematic.}
\[
\begin{array}{c|c|c|c}
\text{base residue}&U_{2r}(A)&U_{2r}(B)&
U_{2r}(A)\ \dot\cup\ (r+U_{2r}(B))\\ \hline
a&\{a,a+r\}&\varnothing&\{a,a+r\}\\
b&\varnothing&\{b,b+r\}&\{b,b+r\}\\
u\notin\{a,b\}&\{p_u\}&\{p_u\}&\{p_u,p_u+r\}
\end{array}
\]
\end{minipage}}
\caption{The local $R^*$ port controls every tail-prefix lift modulo $2r$.}
\label{fig:prefix-lifts}
\end{figure}

\subsection{Hamiltonicity and balance}

If $r=4h+1$, then $(r+1)/2$ is odd and
\begin{equation}\label{eq:S-congruences}
  S\equiv r\pmod{2r},\qquad 2S\equiv2r\pmod{4r}.
\end{equation}

\begin{proposition}\label{prop:abab-hamilton}
The word $\mathcal W=ABAB$ is a Hamilton step word on $\Z_{4r}$.
\end{proposition}

\begin{proof}
The tail prefix sums in the first $AB$ half, modulo $2r$, are
\[
 U_{2r}(A)\ \dot\cup\ (S+U_{2r}(B))=\Z_{2r}
\]
by \eqref{eq:S-congruences} and Lemma~\ref{lem:prefix-lifts}.  The second
copy begins after displacement $2S\equiv2r\pmod{4r}$ and therefore gives
the other lift of every residue modulo $2r$.  All $4r$ tail prefix sums are
distinct modulo $4r$.  Finally $4S$ is divisible by $4r$, since $r$ is
odd, so the walk closes and is Hamilton.
\end{proof}

For $s\in[r]$, let $p_A(s)$ and $p_B(s)$ be the tail prefix sums of its
occurrences in $A$ and $B$. Comparing the two port blocks gives
\begin{equation}\label{eq:relative-prefix}
 p_B(s)-p_A(s)=
 \begin{cases}
   b+r,&s=a,\\
   -(a+r),&s=b,\\
   b-a,&s=r,\\
   0,&s\notin\{a,b,r\}.
 \end{cases}.
\end{equation}
Since $a,b,r$ are all odd, every value in \eqref{eq:relative-prefix} is even.

\begin{proposition}\label{prop:abab-balance}
For every $s\in[r]$, the four occurrences of $s$ in $ABAB$ have tails in
the four distinct residue classes modulo $4$.
\end{proposition}

\begin{proof}
The four absolute tail prefix sums are
\begin{equation}\label{eq:four-tail-prefixes}
 p_A(s),\quad S+p_B(s),\quad2S+p_A(s),\quad3S+p_B(s).
\end{equation}
Here $S$ is odd, while $p_A(s)$ and $p_B(s)$ have the same parity by
\eqref{eq:relative-prefix}.  The first two prefixes therefore have opposite
parity.  Moreover $2S\equiv2\pmod4$, so the last two give the other residue
within each parity.  The four tails are exactly $0,1,2,3$ modulo $4$.
\end{proof}

Figure~\ref{fig:abab-balance} summarizes the residue pattern used above.

\begin{figure}[ht]
\centering
\fbox{\begin{minipage}{0.90\linewidth}
\centering
\textbf{A fixed step across the four blocks.}
\[
\begin{array}{c|cccc}
\text{block}&\text{first }A&\text{first }B&\text{second }A&\text{second }B\\ \hline
\text{tail}&p_A&S+p_B&2S+p_A&3S+p_B\\
\text{parity}&\epsilon&1-\epsilon&\epsilon&1-\epsilon\\
\text{second-half shift}&0&0&2&2
\end{array}
\]
Here \(\epsilon\in\{0,1\}\) is defined by
\(\epsilon\equiv p_A(s)\pmod2\). Lemma~\ref{lem:prefix-lifts} provides
the Hamiltonicity input; the parity pattern gives balance.
\end{minipage}}
\caption{Balance of the $ABAB$ word.}
\label{fig:abab-balance}
\end{figure}

The preceding two propositions show that the $ABAB$ word is a Hamilton step word and that its four occurrences of each step have tails in the four distinct residue classes modulo $4$. Hence it determines a $4$-layer balanced Hamilton starter in the sense of Definition~\ref{def:starter}.

\begin{theorem}\label{thm:four-exact}
For every $r\equiv1\pmod4$ with $r\geq9$, the digraph $\Cfour{r}$ has a
$4$-layer balanced Hamilton starter.  Its translates by $4$ form a Hamilton
decomposition, and every developed cycle has winding number $(r+1)/2$.
\end{theorem}

\begin{proof}
Since \(r\equiv1\pmod4\) and \(r\ge9\), Theorem~\ref{thm:rstar} gives a directed \(R^*\)-terrace of \(\Z_r\). By Lemma~\ref{lem:odd-normalization} and Lemma~\ref{lem:symmetries}, after multiplying by a suitable unit we may choose the distinguished adjacent differences \(a\) and \(b\) to have odd positive representatives. From this normalized terrace we form the two permutations
\[
A=(a,r,b,D'),\qquad B=(b,r,a,D')
\]
of \([r]\), and set \(\mathcal W=ABAB\).

By Proposition~\ref{prop:abab-hamilton}, \(\mathcal W\) is a Hamilton step word on \(\Z_{4r}\), and hence determines a Hamilton cycle \(H\) in \(\Cfour{r}\). Since both \(A\) and \(B\) are permutations of \([r]\), every step \(s\in[r]\) occurs exactly once in each block and therefore exactly four times in \(\mathcal W\). Proposition~\ref{prop:abab-balance} shows that the tails of these four occurrences lie in the four distinct residue classes modulo \(4\). Consequently, for every \((j,s)\in\Z_4\times[r]\), the cycle \(H\) contains exactly one arc of step \(s\) whose tail is congruent to \(j\pmod4\). Thus \(H\) is a \(4\)-layer balanced Hamilton starter.

Finally, Lemma~\ref{lem:development} implies that the \(r\) translates
\[
H,\ H+4,\ \ldots,\ H+4(r-1)
\]
are pairwise arc-disjoint Hamilton cycles whose union is \(A(\Cfour{r})\), and therefore form a Hamilton decomposition. Each translate has the same step sequence as \(H\), so Proposition~\ref{prop:winding} gives winding number \((r+1)/2\) for every developed cycle. This proves the theorem.
\end{proof}

\section{Chain compatibility and the four-layer path number}
\label{sec:chain-compatibility}

We now use the four-layer Hamilton decomposition to obtain an optimal path
decomposition.  The extra ingredient is a transversal of the developed cycles
whose arcs form one simple directed path.  This condition is not a formal
consequence of a Hamilton decomposition, so we begin with a precise definition
and an equivalent development-colouring formulation.

\subsection{Development colours and chain compatibility}

Let $H$ be a $q$-layer balanced Hamilton starter in $\Cqr{q}{r}$, and write
$H_k=H+kq$ for $k\in\Z_r$.  Since the cycles $H_k$ partition the arc set,
every arc of $\Cqr{q}{r}$ has a unique \emph{development colour}, namely the
unique $k$ for which it belongs to $H_k$.

\begin{definition}\label{def:chain-compatible}
The starter $H$ is \emph{chain-compatible} if there are arcs
\[
  e_k=u_k\longrightarrow v_k\in H_k\qquad(k\in\Z_r)
\]
and a permutation $\sigma$ of $\Z_r$ such that
\[
  v_{\sigma(i)}=u_{\sigma(i+1)}\qquad(0\leq i<r-1)
\]
and the resulting $r+1$ vertices are pairwise distinct.
\end{definition}

Equivalently, the selected arcs form a simple directed path of length $r$ which is rainbow for the development colouring; such a path will be called a \emph{compatible deletion chain} for $H$.

For an explicit colour formula, fix $j\in\Z_q$ and $s\in[r]$.  Balance gives
a unique $\beta_j(s)\in\Z_r$ such that $H$ contains the arc of step $s$ with
tail $j+q\beta_j(s)$. If $x=j+qt$ with $t\in\Z_r$, define
\begin{equation}\label{eq:development-colour}
  \chi(x,s):=t-\beta_j(s)\pmod r.
\end{equation}

\begin{lemma}\label{lem:development-colour}
The arc $x\to x+s$ belongs to $H_{\chi(x,s)}$.  Consequently, a simple
directed path of length $r$ is a compatible deletion chain exactly when its
$r$ arc colours under \eqref{eq:development-colour} are pairwise distinct.
\end{lemma}

\begin{proof}
The translate $H_k$ contains $x\to x+s$ exactly when $H$ contains
$(x-kq)\to(x-kq)+s$.  Writing $x=j+qt$, this is equivalent to
$t-k=\beta_j(s)$ in $\Z_r$, or $k=t-\beta_j(s)$.  The second assertion follows
because there are exactly $r$ development colours.
\end{proof}

The same formula yields a convenient partial-sum criterion.  Let
$s_0,\ldots,s_{r-1}\in[r]$ and put
\[
  z_i=z_0+\sum_{h<i}s_h\pmod{qr}\quad(0\leq i\leq r),
  \qquad
  z_i=j_i+qt_i,\quad j_i\in\Z_q,\quad t_i\in\Z_r\quad(0\leq i<r).
\]

\begin{corollary}\label{cor:partial-sum-criterion}
The step word $(s_0,\ldots,s_{r-1})$ gives a compatible deletion chain from
$z_0$ if and only if
\begin{enumerate}
  \item $z_0,z_1,\ldots,z_r$ are pairwise distinct in $\Z_{qr}$; and
  \item the residues
  \[
    t_i-\beta_{j_i}(s_i)\qquad(0\leq i<r)
  \]
  are pairwise distinct in $\Z_r$.
\end{enumerate}
\end{corollary}

\begin{proof}
The first condition says that the directed walk is a simple $r$-arc path.
The second says, by Lemma~\ref{lem:development-colour}, that its arcs use all
development colours exactly once.
\end{proof}

\subsection{The uniform \texorpdfstring{step-$2$}{step-2} chain for \texorpdfstring{$ABAB$}{ABAB}}

Assume $r\equiv1\pmod4$ and $r\geq9$, and retain the notation of
Section~\ref{sec:four}:
\[
  A=(a,r,b,D'),\qquad B=(b,r,a,D'),\qquad
  \mathcal W=ABAB,\qquad S=\frac{r(r+1)}2.
\]
Let $H$ denote the $4$-layer balanced Hamilton starter determined by $\mathcal W$. The three values $a,b,r$ are odd, so the even step $2$ lies in $D'$.  Its
position within $A$ and $B$ is the same.  Moreover, the tail prefix sum immediately preceding it is the same in both blocks, because the two port blocks have the common sum $a+b+r$ and are followed by the same ordered suffix. In the notation above, let $p:=p_A(2)=p_B(2)$. The four tails of step $2$ in $H$ are therefore
\begin{equation}\label{eq:step2-tails}
  p,\quad p+S,\quad p+2S,\quad p+3S\pmod{4r}.
\end{equation}

\begin{lemma}\label{lem:abab-chain}
For every $r\equiv1\pmod4$ with $r\geq9$, the path
\begin{equation}\label{eq:four-chain}
  0\longrightarrow2\longrightarrow4\longrightarrow\cdots\longrightarrow2r
\end{equation}
is rainbow for the development colouring of the $ABAB$ starter.  Hence the
starter is chain-compatible.
\end{lemma}

\begin{proof}
Because $r\equiv1\pmod4$, the integer $(r+1)/2$ is odd, and hence so is $S=r(r+1)/2$. Therefore the four tails in \eqref{eq:step2-tails} occupy the four residue classes modulo $4$. Let $X$ be the one congruent to $0$, and write $X=4\tau$ with $\tau\in\Z_r$. By \eqref{eq:S-congruences}, the occurrence two blocks later has tail $Y:=X+2S\equiv X+2r\pmod{4r}$. Since $X\equiv0\pmod4$ and $r\equiv1\pmod4$, we have $Y\equiv2\pmod4$.

The $i$th arc of \eqref{eq:four-chain} has tail $2i$. If $i$ is even, its tail residue is $0$, so it is obtained by translating the starter arc at $X$; hence $X+4k_i\equiv2i\pmod{4r}$ and $k_i=i/2-\tau$ in $\Z_r$. If $i$ is odd, the relevant starter arc is the one at $Y=X+2r$, and $k_i=(2i-X-2r)/4=(i-r)/2-\tau\pmod r$. Since $r$ is odd, $2$ is a unit in $\Z_r$; write $2^{-1}$ for its inverse. In the odd case, $(i-r)/2$ is an integer, and $2\cdot\frac{i-r}{2}=i-r\equiv i\pmod r$, hence $\frac{i-r}{2}\equiv2^{-1}i\pmod r$. The even case similarly gives $i/2\equiv2^{-1}i\pmod r$. Thus both cases are described by
\begin{equation}\label{eq:four-colour-formula}
  k_i=2^{-1}i-\tau\pmod r.
\end{equation}
The map $i\mapsto2^{-1}i-\tau$ is a bijection on $\Z_r$, since $2^{-1}$ is a unit; hence the colours $k_i$ run through every element of $\Z_r$ exactly once. The integers $0,2,\ldots,2r$ are distinct and lie in $[0,4r-1]$. Therefore \eqref{eq:four-chain} is a simple rainbow path of $r$ arcs. Apply Lemma~\ref{lem:development-colour}.
\end{proof}

\subsection{Exact path number in the four-layer family}

Combining the Hamilton decomposition from Theorem~\ref{thm:four-exact} with the compatible deletion chain from Lemma~\ref{lem:abab-chain} yields the exact path number.

\begin{theorem}\label{thm:four-path}
For every $r\equiv1\pmod4$ with $r\geq9$,
\begin{equation}\label{eq:four-path-number}
  \pn(\Cfour{r})=r+1.
\end{equation}
\end{theorem}

\begin{proof}
Let $H$ be the $ABAB$ starter constructed in the proof of
Theorem~\ref{thm:four-exact}.  By Lemma~\ref{lem:abab-chain}, this starter
is chain-compatible.  Choose one deletion-chain arc $e_k$ from each developed
Hamilton cycle $H_k$.  Then each $H_k-e_k$ is a
simple directed Hamilton path, and the deleted arcs themselves form one
simple directed path of $r$ arcs.  Since the developed cycles are pairwise
arc-disjoint, the $r$ remaining Hamilton paths share no arcs with one another
or with the deletion path, and together they cover every arc of the digraph. Hence
$\pn(\Cfour{r})\leq r+1$.

On the other hand, $\Cfour{r}$ has $4r^2$ arcs and a simple directed path has
at most $4r-1$ arcs.  Therefore
\[
  \pn(\Cfour{r})
  \geq\left\lceil\frac{4r^2}{4r-1}\right\rceil
  =\left\lceil r+\frac{r}{4r-1}\right\rceil
  =r+1.
\]
The two bounds prove \eqref{eq:four-path-number}.
\end{proof}

\section{The asymptotic three-layer construction}\label{sec:three}

Let $r$ be odd and let $T=(t_0,\ldots,t_{r-2})$ be a directed rotational
terrace with difference word $D=(d_0,\ldots,d_{r-2})$.  Put
\begin{equation}\label{eq:m}
  m:=\frac{r-1}{2}.
\end{equation}
We first lift $T$ to $\Z_{qr}$ for a general number $q$ of layers.  This
separates the algebra of the lift from the special three-layer trade.

\subsection{A general terrace lift}

Assume $q\geq3$ and $\gcd(q,r)=1$.  Every nonzero $x\in\Z_r$ has $q$ lifts
$x,x+r,\ldots,x+(q-1)r$ in $\Z_{qr}$.  Define a spanning directed graph
$F_q(T)$ by its outgoing-step function $D_T:\Z_{qr}\to[r]$, where
$D_T(x)=s$ means that the outgoing arc at $x$ is $x\to x+s$, and set
\begin{equation}\label{eq:q-lift}
 D_T(t_i+\ell r)=d_i
 \quad(0\leq\ell<q),
 \qquad
 D_T(\ell r)=r.
\end{equation}

\begin{proposition}\label{prop:q-lift}
If $\gcd(q,r)=1$, then $F_q(T)$ is a $q$-layer balanced directed $1$-factor of
$\Cqr{q}{r}$: every vertex has indegree and outdegree one, and for every
$(j,s)\in\Z_q\times[r]$ exactly one arc of step $s$ has tail congruent to
$j$ modulo $q$.
\end{proposition}

\begin{proof}
For each base edge $t_i\to t_{i+1}$, the $q$ arcs in
\eqref{eq:q-lift} map the $q$ lifts of $t_i$ bijectively to the $q$ lifts of
$t_{i+1}$. The lifts of \(0\) form a directed \(q\)-cycle under step \(r\), since \(r\) has additive order \(q\) in \(\mathbb{Z}_{qr}\). Hence every vertex has one incoming and one outgoing arc.

For a fixed nonzero step $d_i$, its $q$ tails are
$t_i,t_i+r,\ldots,t_i+(q-1)r$.  Since $r$ is a unit modulo $q$, these tails
occupy all $q$ residue classes modulo $q$.  The same statement holds for
the step $r$ at tails $0,r,\ldots,(q-1)r$. Thus the factor is $q$-layer balanced.
\end{proof}

The cycle structure is controlled by the displacement accumulated during
one base round.  Since $D$ is a permutation of $1,\ldots,r-1$,
\begin{equation}\label{eq:base-round-sum}
  \sum_{i=0}^{r-2}d_i=\frac{r(r-1)}2=mr.
\end{equation}

\begin{proposition}\label{prop:cycle-count}
The nonzero part of $F_q(T)$ consists of $\gcd(q,m)$ cycles, each of length
\(
  \frac{q(r-1)}{\gcd(q,m)}.
\)
The zero lifts form one directed $q$-cycle.
\end{proposition}

\begin{proof}
Traversing all \(r-1\) edges of the base terrace changes a lift by the total displacement \(mr\). Hence, in the lift coordinate \(\ell\in\Z_q\), one complete base round induces the translation \(\ell\longmapsto \ell+m\). Thus the orbits of the translation \(\ell\mapsto\ell+m\) on \(\Z_q\) are in one-to-one correspondence with the directed cycles in the nonzero part of \(F_q(T)\). This translation has \(\gcd(q,m)\) orbits, each of length \(q/\gcd(q,m)\). Each orbit therefore requires \(q/\gcd(q,m)\) base rounds, and since each base round contains \(r-1\) vertices, the corresponding directed cycle has length \(q(r-1)/\gcd(q,m)\). Hence the nonzero part of \(F_q(T)\) consists of \(\gcd(q,m)\) directed cycles of the asserted length. The assertion about the zero lifts follows directly from the last rule \eqref{eq:q-lift}.
\end{proof}

\begin{corollary}\label{cor:three-cycle-type}
If $r\equiv5\pmod6$, then $F_3(T)$ has exactly two cycles: one cycle $L$
of length $3(r-1)$ on the nonzero lifts, and the $3$-cycle
\begin{equation}\label{eq:zero-cycle}
  Z=(0,r,2r).
\end{equation}
\end{corollary}

\begin{proof}
Here $m=(r-1)/2\equiv2\pmod3$, so $\gcd(3,m)=1$.  Apply
Proposition~\ref{prop:cycle-count}.
\end{proof}

For the remainder of this section, assume $r\equiv5\pmod6$. In particular,
$\gcd(3,r)=1$.

\subsection{A four-arc balanced trade}

To turn $F_3(T)$, which has the two cycles described in Corollary~\ref{cor:three-cycle-type}, into a Hamilton cycle, we introduce a local four-arc trade that preserves both the directed $1$-factor property and $3$-layer balance.

Let $A,B,C\in\{1,\ldots,r-1\}$ and define four base points in $\Z_r$ by
\begin{equation}\label{eq:general-port}
  P=C-A-B,\qquad Q=C-B,\qquad M=-A,\qquad W=C-A.
\end{equation}
Here and below, to say that a terrace contains a directed path means that the
vertex sequence of the path occurs consecutively, with the displayed
orientation, in the cyclic ordering of the terrace. Suppose that the terrace
contains the directed path $P\to Q\to C$ and the edge $M\to W$. Their
differences are $A,B,C$, respectively. We now work in the three-layer lift \(F_3(T)\subseteq\Cthree{r}\). For
$j\in\Z_3$, let $\widetilde{x}_j$ denote the unique lift of $x\in\Z_r$ that
is congruent to $j$ modulo $3$; uniqueness follows from the Chinese remainder
theorem, since $\gcd(3,r)=1$.  At a fixed residue $j$, put
\begin{equation}\label{eq:trade-tails}
  x_1=\widetilde P_j,\quad x_2=\widetilde Q_j,\quad
  x_3=\widetilde M_j,\quad z=\widetilde0_j.
\end{equation}
The old steps at those tails are $(A,B,C,r)$.  Replace them by
\begin{equation}\label{eq:trade}
  (A,B,C,r)\longmapsto(B,r,A,C).
\end{equation}

\begin{lemma}\label{lem:head-identities}
If
\begin{equation}\label{eq:trade-congruences}
  A\equiv r\pmod3,\qquad B\equiv C\pmod3,
\end{equation}
then the new heads in \eqref{eq:trade} are a permutation of the old heads.
More precisely,
\begin{equation}\label{eq:head-table-identities}
  x_1+B=x_3+C,\quad x_2+r=x_1+A,\quad
  x_3+A=z+r,\quad z+C=x_2+B
\end{equation}
in $\Z_{3r}$.
\end{lemma}

\begin{proof}
Modulo $r$, the four equalities reduce, in order, to
\[
 P+B=M+C,\quad Q=P+A,\quad M+A=0,\quad C=Q+B,
\]
which follow from \eqref{eq:general-port}.  Modulo $3$, all tails in
\eqref{eq:trade-tails} are $j$; the required equality of head residues is
exactly $B\equiv C$ for the first and fourth pairs and $r\equiv A$ for the
second and third.  Since $r$ and $3$ are coprime, equality modulo both
implies equality in $\Z_{3r}$.
\end{proof}

\begin{proposition}\label{prop:trade-balance}
Under \eqref{eq:trade-congruences}, the four-arc trade preserves both the
directed $1$-factor property and $3$-layer balance.
\end{proposition}

\begin{proof}
The trade changes exactly one outgoing arc at each of four tails, so every
outdegree remains one.  Lemma~\ref{lem:head-identities} says that the four
old heads and four new heads agree as multisets, so every indegree also
remains one.  All four changed tails have the same residue $j$ modulo $3$,
and the old and new step multisets are both $\{A,B,C,r\}$.  Hence for each
step the number of occurrences at each tail residue is unchanged.
\end{proof}

Table~\ref{tab:trade} records the corresponding head permutation.

\begin{table}[ht]
\centering
\caption{The four-arc trade as an exact head permutation.}
\label{tab:trade}
\begin{tabular}{c c c c}
\toprule
tail&old step&old head&new step/new head\\
\midrule
$x_1$&$A$&$x_1+A=x_2+r$&$B$, head $x_1+B=x_3+C$\\
$x_2$&$B$&$x_2+B=z+C$&$r$, head $x_2+r=x_1+A$\\
$x_3$&$C$&$x_3+C=x_1+B$&$A$, head $x_3+A=z+r$\\
$z$&$r$&$z+r=x_3+A$&$C$, head $z+C=x_2+B$\\
\bottomrule
\end{tabular}
\end{table}

\subsection{The cyclic-order splice}

For distinct vertices $u,v,w\in V(L)$, say that $u,v,w$ occur in this
cyclic order on $L$ if, starting at $u$ and following the orientation of
$L$, one encounters $v$ before $w$.  For $x\in V(L)$, write $\succL(x)$
for the successor of $x$ on $L$.

\begin{lemma}\label{lem:splice}
Assume $r\equiv5\pmod6$.  If $x_1,x_3,x_2$ occur in this cyclic order
on $L$, then the factor obtained from $F_3(T)$ by the trade
\eqref{eq:trade} is a Hamilton cycle on $\Z_{3r}$.
\end{lemma}

\begin{proof}
Cut $L$ immediately after $x_1,x_3,x_2$ and cut the zero cycle after $z$.
Let $L_{13}$ be the directed segment from $\succL(x_1)$ to $x_3$,
$L_{32}$ the segment from $\succL(x_3)$ to $x_2$, and $L_{21}$ the segment
from $\succL(x_2)$ to $x_1$.
Let $Z_z$ be the segment beginning at the old successor of $z$ and ending
at $z$.  The cyclic-order hypothesis makes the four segments pairwise
vertex-disjoint, and together they cover all vertices.

The head identities show that the new cycle is
\begin{equation}\label{eq:spliced-cycle}
 x_1\longrightarrow L_{32}\longrightarrow x_2
 \longrightarrow L_{13}\longrightarrow x_3
 \longrightarrow Z_z\longrightarrow z
 \longrightarrow L_{21}\longrightarrow x_1.
\end{equation}
Every segment is traversed once, so \eqref{eq:spliced-cycle} is Hamilton.
\end{proof}

\subsection{The fixed-port construction}

The fixed parameters are
\begin{equation}\label{eq:fixed-ABC}
  A=r-3,\qquad B=1,\qquad C=7.
\end{equation}
Equation \eqref{eq:general-port} then yields
\begin{equation}\label{eq:fixed-PQMW}
  (P,Q,M,W)=(9,6,3,10).
\end{equation}
Thus the required terrace incidences are supplied by the single path
\begin{equation}\label{eq:fixed-path}
  9\to6\to7\to5\to8\to3\to10.
\end{equation}

\begin{table}[ht]
\centering
\caption{Vertices and colours of the prescribed port path.}
\label{tab:fixed-path}
\begin{tabular}{c c c c c c c}
\toprule
edge&$9\to6$&$6\to7$&$7\to5$&$5\to8$&$8\to3$&$3\to10$\\
\midrule
colour $y-x$&$r-3$&$1$&$r-2$&$3$&$r-5$&$7$\\
role&$A$&$B$&separator&separator&separator&$C$\\
\bottomrule
\end{tabular}
\end{table}

For $r\equiv5\pmod6$ with $r\geq11$, the seven displayed vertices are
nonzero and pairwise distinct, and the six displayed colours are also
nonzero and pairwise distinct.  Hence the path in Table~\ref{tab:fixed-path}
is a legitimate rainbow path in $\Gamma_r$.  Given a terrace containing
\eqref{eq:fixed-path}, fix $j\in\Z_3$.  We call the specialization of the
four-arc trade with $(A,B,C)=(r-3,1,7)$ at tail residue $j$ the
\emph{fixed-port trade at residue $j$}, and the resulting terrace-lift
construction the \emph{fixed-port construction at residue $j$}. When the
residue is immaterial, we omit the qualifier ``at residue $j$''.

\begin{lemma}\label{lem:fixed-order}
Let $r\equiv5\pmod6$, and suppose $T$ contains
\eqref{eq:fixed-path}.  For any fixed tail residue $j$ modulo $3$, the lifts
$x_1,x_3,x_2$ of $P=9$, $M=3$, and $Q=6$, respectively, occur in this
cyclic order on $L$.
\end{lemma}

\begin{proof}
The five-edge segment from $9$ to $3$ has positive displacement
\begin{equation}\label{eq:five-displacement}
 (r-3)+1+(r-2)+3+(r-5)=3r-6\equiv0\pmod3.
\end{equation}
Therefore the lift of $3$ with the same tail residue occurs exactly five
edges after the chosen lift of $9$.

After the first edge $9\to6$, the tail residue changes by
$r-3\equiv2\pmod3$.  Since $r\equiv5\pmod6$, we have
\[
  m=\frac{r-1}{2}\equiv2\pmod3
  \qquad\text{and}\qquad
  r\equiv2\pmod3.
\]
Thus a complete base round changes the actual residue modulo $3$ by
\[
  mr\equiv2\cdot2\equiv1\pmod3.
\]
The lift of $6$ with the original residue is therefore reached exactly one
base round after the first occurrence of $6$.  Its directed distance from
$x_1$ is $1+(r-1)=r$.  Since $0<5<r$, starting at $x_1$ one encounters
$x_3$ before $x_2$.  Thus $x_1,x_3,x_2$ occur in the claimed cyclic order
on $L$.
\end{proof}

The only remaining input is the existence of a directed rotational terrace containing \eqref{eq:fixed-path}. Section~\ref{sec:prescribed} derives this for all sufficiently large \(r\) from the rainbow Hamilton-path theorem of M\"uyesser and Pokrovskiy~\cite[Theorem~6.9]{MuyesserPokrovskiy2025}.

\begin{theorem}
\label{thm:three-asymptotic}
There exists a non-explicit constant $r_0$ such that, for every
\[
  r\equiv5\pmod6,\qquad r\geq r_0,
\]
and every $j\in\Z_3$, the fixed-port construction at residue $j$ yields a
$3$-layer balanced Hamilton starter of the digraph $\Cthree{r}$.  Its
translates by $3$ form a Hamilton decomposition, and every developed cycle
has winding number $(r+1)/2$.
\end{theorem}

\begin{proof}
For $r\geq11$, \eqref{eq:fixed-path} is a fixed rainbow directed path of
length $6$ in $\Gamma_r$.  Hence Corollary~\ref{cor:prescribed-path}, applied
with $\ell=6$, gives a directed rotational terrace containing
\eqref{eq:fixed-path} for all sufficiently large $r\equiv5\pmod6$.
Fix $j\in\Z_3$ and form the $3$-layer balanced directed
$1$-factor $F_3(T)$.  By Corollary~\ref{cor:three-cycle-type}, before the
trade this factor has the long cycle $L$ and the $3$-cycle $Z$.  Perform the
fixed-port trade at residue $j$.  The parameters \eqref{eq:fixed-ABC} satisfy
$A\equiv r\pmod3$ and $B\equiv C\pmod3$, so
Proposition~\ref{prop:trade-balance} shows that the trade preserves both the
directed $1$-factor property and $3$-layer balance.
Lemma~\ref{lem:fixed-order} supplies the hypothesis of
Lemma~\ref{lem:splice}, which makes the traded factor Hamilton.  Finally,
Lemma~\ref{lem:development} yields the Hamilton decomposition, while
Proposition~\ref{prop:winding} gives the stated winding number.
\end{proof}

\section{Chain compatibility in the three-layer fixed-port construction}
\label{sec:three-chain}

Unlike in the four-layer case, no information about the position of step \(2\) in the terrace is needed; it suffices that the fixed-port trade does not alter any of the three arcs of step \(2\).

\begin{lemma}\label{lem:untouched-step}
Let $q\geq3$ and $r\geq1$ be coprime integers, and let $H$ be a $q$-layer balanced Hamilton starter with the following
property for some step $s$: the $q$ occurrences of $s$ have tails
\[
  b,b+r,\ldots,b+(q-1)r
\]
for a fixed $b\in\Z_r$.  If $\gcd(qr,s)=1$, then
\begin{equation}\label{eq:general-chain}
  0\longrightarrow s\longrightarrow2s\longrightarrow\cdots\longrightarrow rs
  \quad\text{in }\Z_{qr}
\end{equation}
is a compatible deletion chain.
\end{lemma}

\begin{proof}
The vertices in \eqref{eq:general-chain} are pairwise distinct because $s$
has additive order $qr>r$ in $\Z_{qr}$.  Let $X_i$ be the unique lift of $b$
which has the same residue modulo $q$ as the tail $is$; uniqueness holds
because $r$ is a unit modulo $q$.  If the $i$th path arc belongs to $H_{k_i}$,
then
\[
  X_i+qk_i=is\pmod{qr}.
\]
Reducing modulo $r$ and using $X_i\equiv b\pmod r$ gives
\begin{equation}\label{eq:general-chain-colours}
  k_i=q^{-1}(is-b)\pmod r.
\end{equation}
Since \(q\) and \(s\) are units modulo \(r\), the map \(i\longmapsto q^{-1}(is-b)\) is a bijection on \(\Z_r\). Hence the colours \(k_i\) are pairwise distinct for \(0\le i<r\). Thus the simple path is rainbow for the development colouring, and
Lemma~\ref{lem:development-colour} proves chain compatibility.
\end{proof}

\begin{proposition}
\label{prop:three-chain}
Let $r\equiv5\pmod6$ with $r\geq11$.  Every $3$-layer balanced Hamilton
starter obtained from the fixed-port construction at any residue
$j\in\Z_3$, with
\[
  (A,B,C)=(r-3,1,7),
\]
is chain-compatible.  A deletion chain is
\begin{equation}\label{eq:three-chain}
  0\longrightarrow2\longrightarrow4\longrightarrow\cdots\longrightarrow2r.
\end{equation}
\end{proposition}

\begin{proof}
In the untraded terrace lift, the three occurrences of every nonzero step
$s$ have tails at the three lifts of the base tail of the unique terrace edge
of difference $s$.  The fixed-port trade changes only the four steps
\[
  r-3,\quad1,\quad7,\quad r.
\]
For $r\geq11$, step $2$ is not among them.  Hence its three tails remain
$b,b+r,b+2r$ for some $b\in\Z_r$.  Since $r$ is odd and not divisible by
$3$, one has $\gcd(3r,2)=1$.  Lemma~\ref{lem:untouched-step}, with $q=3$ and
$s=2$, gives \eqref{eq:three-chain}.  Explicitly, its colours are
\begin{equation}\label{eq:three-colour-formula}
  k_i=3^{-1}(2i-b)\pmod r,
\end{equation}
which form a permutation of $\Z_r$.
\end{proof}

\begin{theorem}\label{thm:three-layer-chain-compatible}
There exists a non-explicit constant \(r_0\) such that, for every
\(r\equiv5\pmod6\) with \(r\geq r_0\), the digraph \(\Cthree{r}\)
admits a chain-compatible $3$-layer balanced Hamilton starter.
\end{theorem}
\begin{proof}
This follows immediately from Theorem~\ref{thm:three-asymptotic} and
Proposition~\ref{prop:three-chain}, after increasing \(r_0\) if necessary.
\end{proof}

\begin{corollary}\label{cor:three-layer-path-number}
After increasing the non-explicit constant \(r_0\) if necessary,
\[
  \pn(\Cthree{r})=r+1
\]
for every \(r\equiv5\pmod6\) with \(r\geq r_0\).
\end{corollary}
\begin{proof}
By Theorem~\ref{thm:three-layer-chain-compatible}, choose a chain-compatible
$3$-layer balanced Hamilton starter.  By Lemma~\ref{lem:development}, its $r$
translates form a Hamilton decomposition.  Deleting one chain arc from each
developed Hamilton cycle gives $r$ Hamilton paths, while the deleted
arcs form one further simple path.  Hence \(\pn(\Cthree{r})\leq r+1\).
Since \(\Cthree{r}\) has \(3r^2\) arcs and every simple path has at most
\(3r-1\) arcs,
\[
  \pn(\Cthree{r})
  \geq
  \left\lceil\frac{3r^2}{3r-1}\right\rceil
  =r+1.
\]
Thus \(\pn(\Cthree{r})=r+1\).
\end{proof}

The numerical equality in Corollary~\ref{cor:three-layer-path-number} also follows asymptotically from the general robust-expander theorem discussed in Appendix~\ref{app:robust-expansion}; the content here is the specified cyclic Hamilton decomposition together with its compatible deletion chain.

\section{Prescribed-path extension for directed rotational terraces}
\label{sec:prescribed}

We now derive the prescribed-path extension needed in the three-layer
construction from a rainbow Hamilton-path theorem of M\"uyesser and
Pokrovskiy. We use the colour convention already fixed for $\Gamma_r$: an
arc $u\to v$ has colour $v-u\in\Z_r$. For a finite set
$C\subseteq\Z_r$, write
\[
  \sum C:=\sum_{c\in C}c\in\Z_r
\]
for its group sum.

In the division digraph notation of M\"uyesser and Pokrovskiy, the colour of
$u\to v$ is $u^{-1}v$; for the additive group $\Z_r$ this is exactly $v-u$.
Thus \cite[Theorem~6.9]{MuyesserPokrovskiy2025} takes the following form in
our notation.

\begin{externalresult}[M\"uyesser--Pokrovskiy, Theorem~6.9]
For all sufficiently large $r$, let $V,C\subseteq\Z_r$ and let
$x,y\in\Z_r$ satisfy
\begin{equation}\label{eq:mp-size}
  |V|+1=|C|\geq r-r^{1/2},
\end{equation}
\begin{equation}\label{eq:mp-endpoints}
  x\neq y,
  \qquad
  x,y\notin V,
\end{equation}
and
\begin{equation}\label{eq:mp-sum}
  \sum C=y-x.
\end{equation}
Then the subdigraph on the vertex set $\{x,y\}\cup V$ consisting of the
arcs $u\to v$ whose colour $v-u$ lies in $C$ contains a directed rainbow
Hamilton path from $x$ to $y$.
\end{externalresult}

The exceptional clause in the original theorem for elementary abelian
$2$-groups does not arise in our application, since $\Z_r$ has odd order.
The following consequence is the form needed in the three-layer construction.
We obtain it by applying the external result to the vertices and colours left
unused by a prescribed path in $\Gamma_r$.

\begin{corollary}
\label{cor:prescribed-path}
Fix $\ell\geq1$.  For every sufficiently large odd $r$, every rainbow
directed path of length $\ell$ in $\Gamma_r$ is contained, with its
orientation and order unchanged, in a rainbow Hamilton cycle using every
nonzero vertex and every nonzero colour.  Equivalently, the cyclic ordering
of this Hamilton cycle is a directed rotational terrace of $\Z_r$ in which
the vertex sequence of the prescribed path occurs consecutively.
\end{corollary}

\begin{proof}
Let
\[
  P=(v_0,v_1,\ldots,v_\ell)
\]
be a prescribed rainbow directed path in $\Gamma_r$, and put
\[
  a:=v_0,
  \qquad
  b:=v_\ell.
\]
Since $P$ is a simple directed path of positive length, $a\neq b$.  Moreover,
all vertices of $P$ are nonzero because $V(\Gamma_r)=\Z_r\setminus\{0\}$.
Write
\[
  V(P):=\{v_0,v_1,\ldots,v_\ell\}
\]
for its vertex set.  Let
\[
  C(P):=\{v_i-v_{i-1}:1\leq i\leq\ell\}
\]
be the arc-colour set of $P$.  The path is rainbow, so
\begin{equation}\label{eq:prescribed-colour-size}
  |C(P)|=\ell.
\end{equation}

We use the vertices and colours not already used by $P$ as the residual
sets
\begin{equation}\label{eq:prescribed-residual-sets}
  V:=(\Z_r\setminus\{0\})\setminus V(P),
  \qquad
  C:=(\Z_r\setminus\{0\})\setminus C(P).
\end{equation}
The path $P$ has $\ell+1$ vertices, while \eqref{eq:prescribed-colour-size}
shows that it uses $\ell$ colours.  Hence
\begin{equation}\label{eq:prescribed-residual-size}
  |V|=r-\ell-2,
  \qquad
  |C|=r-\ell-1=|V|+1.
\end{equation}
Since $\ell$ is fixed, for all sufficiently large $r$ we also have
\[
  |C|=r-\ell-1\geq r-r^{1/2}.
\]
Thus the size condition \eqref{eq:mp-size} is satisfied.  By construction,
$a,b\notin V$, and $a\neq b$, so the endpoint conditions
\eqref{eq:mp-endpoints} are satisfied as well.

It remains to verify the group-sum condition.  The colours along $P$
telescope:
\begin{equation}\label{eq:prescribed-telescope}
  \sum C(P)
  =\sum_{i=1}^{\ell}(v_i-v_{i-1})
  =b-a.
\end{equation}
Because $r$ is odd,
\begin{equation}\label{eq:nonzero-group-sum}
  \sum_{z\in\Z_r\setminus\{0\}}z
  =\frac{r(r-1)}2
  =0
  \qquad\text{in }\Z_r.
\end{equation}
Taking the complement of the colour set in
\eqref{eq:prescribed-residual-sets} and using
\eqref{eq:prescribed-telescope} therefore gives
\begin{equation}\label{eq:prescribed-residual-sum}
  \sum C
  =-\sum C(P)
  =a-b.
\end{equation}

We now apply the external result with
\[
  x:=b,
  \qquad
  y:=a,
\]
and with the residual sets $V$ and $C$ from
\eqref{eq:prescribed-residual-sets}.  Equation
\eqref{eq:prescribed-residual-sum} is exactly the required condition
\[
  \sum C=y-x.
\]
Hence there is a directed rainbow Hamilton path
\[
  Q:b\longrightarrow a
\]
on the vertex set $\{a,b\}\cup V$ using only colours from $C$.  By
\eqref{eq:prescribed-residual-size}, $Q$ has $|V|+1=|C|$ arcs.  Since these
arc colours are pairwise distinct and all lie in $C$, the arc-colour set of
$Q$ is exactly $C$.

The internal vertices of $Q$ are precisely the nonzero vertices not used by
$P$, so $P$ and $Q$ meet only at their endpoints $a$ and $b$.  Concatenating
\[
  P:a\longrightarrow b
  \qquad\text{and}\qquad
  Q:b\longrightarrow a
\]
therefore produces a directed Hamilton cycle on
$\Z_r\setminus\{0\}$.  Its colour set is the disjoint union
\[
  C(P)\,\dot\cup\,C
  =\Z_r\setminus\{0\},
\]
so the cycle is rainbow and uses every nonzero colour exactly once.  The
path $P$ appears in this cycle with exactly its prescribed orientation and
vertex order.  Finally, Proposition~\ref{prop:rainbow} identifies the cyclic
ordering of this rainbow Hamilton cycle with a directed rotational terrace
of $\Z_r$, completing the proof.
\end{proof}

\begin{remark}
The threshold in Corollary~\ref{cor:prescribed-path} is non-explicit because
the sufficiently large hypothesis is inherited from
\cite[Theorem~6.9]{MuyesserPokrovskiy2025}.  
\end{remark}

\section{Conclusion and open problems}\label{sec:conclusion}

We constructed two infinite families of $q$-layer balanced Hamilton starters, each with the arithmetic step-$2$ path as a compatible deletion chain. The four-layer construction is explicit for $r\equiv1\pmod4$, $r\geq9$, while the three-layer construction is asymptotic for $r\equiv5\pmod6$.  As a consequence, deleting the chain arcs gives $\pn(\Cfour{r})=r+1$ and $\pn(\Cthree{r})=r+1$ in the respective ranges.  In both families every developed Hamilton cycle has winding number $(r+1)/2$.  The prescribed-path extension used for $q=3$ follows directly from \cite[Theorem~6.9]{MuyesserPokrovskiy2025}.

The present four-layer $ABAB$ construction does not extend directly to $r\equiv3\pmod4$. Indeed, set $S=r(r+1)/2$. When $r\equiv1\pmod4$, one has $S\equiv r\pmod{2r}$ and $2S\equiv2r\pmod{4r}$; these congruences are used in both the complementary-prefix lift and the four-residue balance arguments. When $r\equiv3\pmod4$, they become $S\equiv0\pmod{2r}$ and $2S\equiv0\pmod{4r}$, so the present Hamiltonicity and balance mechanisms fail. This is a limitation of the $ABAB$ construction rather than an obstruction to the existence of a different four-layer starter. Two natural questions remain.
\begin{enumerate}
\item Determine whether a different four-layer construction exists for $r\equiv3\pmod4$, or whether a genuine structural obstruction occurs in this congruence class.
\item Construct directly a directed rotational terrace containing \eqref{eq:fixed-path}, ideally for every $r\equiv5\pmod6$ above a small explicit threshold.
\end{enumerate}

\appendix
\numberwithin{equation}{section}
\section{Robust expansion and the general path-number theorem}\label{app:robust-expansion}

This appendix records why, for fixed $q\in\{3,4\}$ and all sufficiently
large $r$, the ordinary equality $\pn(\Cqr{q}{r})=r+1$ also follows from
the general robust-expander theory.  It is independent of the balanced
starter and compatible-chain constructions in the main text.

For a digraph $D$ on $n$ vertices and $S\subseteq V(D)$, the
\emph{$\nu$-robust outneighbourhood} of $S$ is
\[
  RN^+_{\nu,D}(S):=
  \{v\in V(D):|N_D^-(v)\cap S|\geq \nu n\}.
\]
The digraph $D$ is a \emph{robust $(\nu,\tau)$-outexpander} if
\[
  |RN^+_{\nu,D}(S)|\geq |S|+\nu n
\]
whenever $\tau n\leq |S|\leq(1-\tau)n$.

We use the following consequence of Gir\~ao--Granet--K\"uhn--Lo--Osthus
\cite[Theorem~5.2]{GiraoGranetKuhnLoOsthus2023}.

\begin{theorem}[Gir\~ao--Granet--K\"uhn--Lo--Osthus]\label{thm:general-robust-path}
Let
\[
  0<\frac1n\ll \nu\ll\tau\leq\frac{\delta}{2}\leq1
  \qquad\text{and}\qquad
  r\geq\delta n.
\]
If $D$ is an $r$-regular digraph on $n$ vertices and $D$ is a robust
$(\nu,\tau)$-outexpander, then
\[
  \pn(D)=r+1.
\]
\end{theorem}

Here $\delta$ is the fixed density parameter in the linear-degree condition
$r\geq\delta n$.  We now verify the robust-expansion hypothesis for the
circulants used in this paper with parameters chosen to fit
Theorem~\ref{thm:general-robust-path} directly.

\begin{proposition}\label{prop:circulant-robust-expansion}
Fix $q\in\{3,4\}$, put $c=1/q$, and set
\[
  \tau:=\frac{c}{4}=\frac1{4q}.
\]
Then $\nu>0$ can be chosen sufficiently small relative to $\tau$ so that,
for all sufficiently large $r$, $\Cqr{q}{r}$ is a robust
$(\nu,\tau)$-outexpander.
\end{proposition}

\begin{proof}
Put $n=qr$, so $r=cn$, and let $S\subseteq\Z_n$ satisfy
\[
  |S|=\alpha n,\qquad \tau\leq\alpha\leq1-\tau.
\]
For $v\in\Z_n$, define
\[
  f_S(v):=\bigl|S\cap\{v-r,\ldots,v-1\}\bigr|.
\]
Thus $f_S(v)=|N^-_{\Cqr{q}{r}}(v)\cap S|$.  Consecutive windows differ in
one deleted and one added vertex, so
$|f_S(v+1)-f_S(v)|\leq1$, while double counting gives
$\sum_v f_S(v)=r|S|$.

Let
\[
  R:=RN^+_{\nu,\Cqr{q}{r}}(S)=\{v:f_S(v)\geq\nu n\},
\]
and suppose for a contradiction that
\begin{equation}\label{eq:app-R-small}
  |R|<|S|+\nu n=(\alpha+\nu)n.
\end{equation}
We choose $\nu<\tau$.  Then $|R|<(1-\tau+\nu)n<n$, so $R\neq\Z_n$.
Set
\[
  t:=\lceil\nu n\rceil,\qquad h:=r-t+1,
  \qquad g(v):=r-f_S(v).
\]
Since $\nu<\tau=c/4$, we have $\nu n<r/4<r$, so $t\leq r$ and hence
$h\geq1$.  The function $g$ is $1$-Lipschitz, and every $v\notin R$ satisfies
$g(v)\geq h$.  Moreover,
\begin{equation}\label{eq:app-g-exact}
  \sum_{v\in\Z_n}g(v)=r(n-|S|).
\end{equation}

Consider a cyclic interval component $C=\{v_1,\ldots,v_\ell\}$ of $R$,
with predecessor $v_0\notin R$.  Since $g(v_0)\geq h$ and $g$ is
$1$-Lipschitz,
\[
  g(v_j)\geq\max\{h-j,0\}\qquad(1\leq j\leq\ell).
\]
A direct summation yields
\[
  \sum_{v\in C}g(v)
  \geq \sum_{j=1}^{\ell}\max\{h-j,0\}
  \geq \frac{h-1}{2}\min\{\ell,h\}.
\]
If the components of $R$ have lengths $\ell_1,\ldots,\ell_k$, then
$\sum_i\min\{\ell_i,h\}\geq\min\{|R|,h\}$.  Adding these component
estimates to the bound $g(v)\geq h$ for $v\notin R$ gives
\begin{equation}\label{eq:app-g-lower}
  \sum_{v\in\Z_n}g(v)
  \geq (n-|R|)h+\frac{h-1}{2}\min\{|R|,h\}.
\end{equation}

For $0\leq x\leq n$, define
\[
  F_h(x):=(n-x)h+\frac{h-1}{2}\min\{x,h\}.
\]
Its two slopes are $-(h+1)/2$ and $-h$, so $F_h$ is strictly decreasing.
By \eqref{eq:app-R-small} and \eqref{eq:app-g-lower},
\begin{equation}\label{eq:app-F-lower}
  \sum_{v\in\Z_n}g(v)>F_h((\alpha+\nu)n).
\end{equation}

It remains to compare the right-hand side with the exact value in
\eqref{eq:app-g-exact}.  Put
\[
  A_n:=\frac hn,\qquad B_n:=\frac{h-1}{n},
  \qquad \lambda:=c-\nu.
\]
Since $h=cn-\lceil\nu n\rceil+1$,
\[
  |A_n-\lambda|\leq\frac1n,
  \qquad |B_n-\lambda|\leq\frac1n.
\]
Define
\[
  \Psi_{n,\nu}(\alpha)
  :=\frac1{n^2}F_h((\alpha+\nu)n)
  =A_n(1-\alpha-\nu)
   +\frac{B_n}{2}\min\{\alpha+\nu,A_n\},
\]
and
\[
  \Phi_\nu(\alpha)
  :=\lambda(1-\alpha-\nu)
   +\frac{\lambda}{2}\min\{\alpha+\nu,\lambda\}.
\]
Because the map $(a,x)\mapsto\min\{a,x\}$ is $1$-Lipschitz in each
variable, the preceding bounds give, uniformly for $\alpha\in[\tau,1-\tau]$,
\begin{equation}\label{eq:app-uniform}
  \sup_{\alpha\in[\tau,1-\tau]}
  |\Psi_{n,\nu}(\alpha)-\Phi_\nu(\alpha)|
  \leq\frac2n\longrightarrow0.
\end{equation}

At $\nu=0$,
\[
  \Phi_0(\alpha)-c(1-\alpha)
  =\frac c2\min\{\alpha,c\}.
\]
Since $\tau=c/4$ and $\alpha\geq\tau$,
\[
  \Phi_0(\alpha)-c(1-\alpha)
  \geq\frac{c^2}{8}=:\eta>0
\]
uniformly on $[\tau,1-\tau]$.  By continuity of
$(\alpha,\nu)\mapsto\Phi_\nu(\alpha)$ on a compact neighbourhood of
$[\tau,1-\tau]\times\{0\}$, $\nu>0$ may be chosen sufficiently small
(relative to $\tau$) so that
\[
  \Phi_\nu(\alpha)>c(1-\alpha)+\frac\eta2
  \qquad\text{for all }\alpha\in[\tau,1-\tau].
\]
Fix such a $\nu$.  By the uniform convergence
\eqref{eq:app-uniform}, for all sufficiently large $r$,
\[
  \Psi_{n,\nu}(\alpha)>c(1-\alpha)
  \qquad\text{for all }\alpha\in[\tau,1-\tau].
\]
Thus
\[
  F_h((\alpha+\nu)n)>c(1-\alpha)n^2=r(n-|S|),
\]
which together with \eqref{eq:app-F-lower} contradicts
\eqref{eq:app-g-exact}.  Hence
$|R|\geq|S|+\nu n$, as required.
\end{proof}

\begin{corollary}\label{cor:general-path-number}
For each fixed $q\in\{3,4\}$, all sufficiently large $r$ satisfy
\[
  \pn(\Cqr{q}{r})=r+1.
\]
\end{corollary}

\begin{proof}
Set $D:=\Cqr{q}{r}$ and $n:=qr$.  The digraph $D$ is $r$-regular: each
vertex has the $r$ outarcs of steps $1,\ldots,r$ and the corresponding $r$
inarcs.  In Theorem~\ref{thm:general-robust-path}, take
\[
  \delta:=\frac1q.
\]
Then $r=n/q=\delta n$.  Proposition~\ref{prop:circulant-robust-expansion}
uses $\tau=1/(4q)$, so
\[
  \tau=\frac1{4q}<\frac1{2q}=\frac\delta2.
\]
Choose $\nu$ in that proposition sufficiently small that $\nu\ll\tau$ as
well, and then take $r$ sufficiently large that $1/n\ll\nu$.  Hence
\[
  0<\frac1n\ll\nu\ll\tau\leq\frac\delta2\leq1,
  \qquad r\geq\delta n,
\]
and all hypotheses of Theorem~\ref{thm:general-robust-path} hold.  Therefore
$\pn(D)=r+1$.
\end{proof}

\section*{Acknowledgements}
This work was supported by the National Natural Science Foundation of China under Grant No.~12071351.


\begin{thebibliography}{99}


\bibitem{AlspachMasonPullman1976}
Brian R. Alspach, David W. Mason, and Norman J. Pullman.
\newblock Path numbers of tournaments.
\newblock \emph{Journal of Combinatorial Theory, Series B},
  20(3):222--228, 1976.
\newblock \url{https://doi.org/10.1016/0095-8956(76)90013-7}.

\bibitem{AlspachPullman1974}
Brian R. Alspach and Norman J. Pullman.
\newblock Path decompositions of digraphs.
\newblock \emph{Bulletin of the Australian Mathematical Society},
  10(3):421--427, 1974.
\newblock \url{https://doi.org/10.1017/S0004972700041101}.


\bibitem{DuHsu1989}
D. Z. Du and D. Frank Hsu.
\newblock On Hamiltonian consecutive-$d$ digraphs.
\newblock \emph{Banach Center Publications}, 25(1):47--55, 1989.
\newblock \url{https://doi.org/10.4064/-25-1-47-55}.

\bibitem{GiraoGranetKuhnLoOsthus2023}
Ant\'onio Gir\~ao, Bertille Granet, Daniela K\"uhn, Allan Lo, and Deryk Osthus.
\newblock Path decompositions of tournaments.
\newblock \emph{Proceedings of the London Mathematical Society},
  126(2):429--517, 2023.
\newblock \url{https://doi.org/10.1112/plms.12480}.

\bibitem{JordonMorris2009}
Heather Jordon and Joy Morris.
\newblock Directed cyclic Hamiltonian cycle systems of the complete symmetric digraph.
\newblock \emph{Discrete Mathematics}, 309(4):784--796, 2009.
\newblock \url{https://doi.org/10.1016/j.disc.2008.01.016}.

\bibitem{KadriSajna2025}
Suzan Kadri and Mateja \v{S}ajna.
\newblock The directed Oberwolfach problem with variable cycle lengths: A recursive construction.
\newblock \emph{Journal of Combinatorial Designs}, 33:239--260, 2025.
\newblock \url{https://doi.org/10.1002/jcd.21967}.

\bibitem{KuhnOsthus2013}
Daniela K\"uhn and Deryk Osthus.
\newblock Hamilton decompositions of regular expanders: A proof of Kelly's conjecture for large tournaments.
\newblock \emph{Advances in Mathematics}, 237:62--146, 2013.
\newblock \url{https://doi.org/10.1016/j.aim.2013.01.005}.

\bibitem{Liu2009}
Guizhen Liu.
\newblock Orthogonal factorizations of digraphs.
\newblock \emph{Frontiers of Mathematics in China}, 4(2):311--323, 2009.
\newblock \url{https://doi.org/10.1007/s11464-009-0011-y}.

\bibitem{LoPatelSkokanTalbot2020}
Allan Lo, Viresh Patel, Jozef Skokan, and John Talbot.
\newblock Decomposing tournaments into paths.
\newblock \emph{Proceedings of the London Mathematical Society},
  121(2):426--461, 2020.
\newblock \url{https://doi.org/10.1112/plms.12328}.

\bibitem{MuyesserPokrovskiy2025}
Alp M\"uyesser and Alexey Pokrovskiy.
\newblock A random Hall--Paige conjecture.
\newblock \emph{Inventiones Mathematicae}, 240:779--867, 2025.
\newblock \url{https://doi.org/10.1007/s00222-025-01328-x}.

\bibitem{Ollis2012}
Matthew A. Ollis.
\newblock A note on terraces for abelian groups.
\newblock \emph{Australasian Journal of Combinatorics}, 52:229--234, 2012.
\newblock \url{https://ajc.maths.uq.edu.au/pdf/52/ajc_v52_p229.pdf}.

\bibitem{PasottiPellegrini2018}
Anita Pasotti and Marco Antonio Pellegrini.
\newblock Cyclic uniform 2-factorizations of the complete multipartite graph.
\newblock \emph{Graphs and Combinatorics}, 34:901--930, 2018.
\newblock \url{https://doi.org/10.1007/s00373-018-1920-x}.

\bibitem{PatelYildiz2026}
Viresh Patel and Mehmet Akif Y\i ld\i z.
\newblock Path decompositions of oriented graphs.
\newblock \emph{European Journal of Combinatorics}, 134:104346, 2026.
\newblock \url{https://doi.org/10.1016/j.ejc.2026.104346}.

\end{thebibliography}
\end{document}